\documentclass[12pt]{amsart}

\usepackage{amsmath}
\usepackage{amssymb}
\usepackage{graphicx}

\newcommand{\C}{{\mathbb C}}       
\newcommand{\R}{{\mathbb R}}       

\newcommand{\N}{{\mathbb N}}       %

\newcommand{\HH}{{\mathcal H}}

\newcommand{\CC}{{\mathcal C}}

\newcommand{\diam}{{\rm diam}}
\newcommand{\dist}{{\rm dist}}
\newcommand{\ra}{\rightarrow}

\newcommand{\rf}[1]{{(\ref{#1})}}

\newcommand{\supp}{{\operatorname{spt}}}

\newcommand{\vphi}{{\varphi}}
\newcommand{\ve}{{\varepsilon}}
\newcommand{\vv}{}

\newcommand{\hdg}{ \mathcal{H}^{d-1} \lfloor_\Gamma}
\newcommand{\hd}{ \mathcal{H}^{d-1} }
\newcommand{\rest}{{\lfloor}}

\newcommand{\re}{{\rm Re}}

\newtheorem{theorem}{Theorem}[section]
\newtheorem{lemma}[theorem]{Lemma}

\newtheorem*{lemma*}{Lemma}

\theoremstyle{definition}

\theoremstyle{remark}
\newtheorem{rem}[theorem]{\bf Remark}

\numberwithin{equation}{section}

\newcommand{\brem}{\begin{rem}}
\newcommand{\erem}{\end{rem}}

\begin{document}

\title[Strong and weak estimates for singular integrals]{Strong and weak type estimates for singular integrals with respect to measures separated by AD-regular boundaries}

\author{V. Chousionis and X. Tolsa}
\thanks{V.C is partially supported by the grant MTM2010-15657 (Spain) and X.T is partially supported by grants 2009SGR-000420 (Generalitat de Catalunya) and MTM2010-16232
(Spain).}

\address{Vasilis Chousionis.  Department of Mathematics, University of Illinois at Urbana-Champaign, 1409 W. Green Street, Urbana, IL 61801, U.S.A. }

\email{vchous@illinois.edu}

\address{Xavier Tolsa. Instituci\'{o} Catalana de Recerca i Estudis Avan\c{c}ats (ICREA) and Departament de Ma\-te\-m\`a\-ti\-ques, Universitat Aut\`onoma de Bar\-ce\-lo\-na.
08193 Barcelona, Catalonia} \email{xtolsa@mat.uab.cat}

\begin{abstract}
We prove weak and strong boundedness estimates for singular integrals in $\R^d$ with respect to $(d-1)$-dimensional measures separated by Ahlfors-David regular boundaries, generalizing and extending results of Chousionis and Mattila. Our proof follows a different strategy based on new  Calder\'on-Zygmund decompositions which can be also used to extend a result of David.
\end{abstract}
\maketitle

\section{Introduction}


A Radon measure on $\R^d$ has $n$-growth if
there exists some constant $c_\mu$ such that $\mu(B(x,r))\leq c_\mu r^{n}$
for all $x\in\R^d$, $r>0$. If there exists some constant $c_\mu$ such that
$$c_\mu^{-1}r^{n}\leq \mu(B(x,r)) \leq c_\mu\,r^{n}\quad\mbox{for all $x\in\supp\mu$,
$0<r\leq\diam(\supp\mu)$},$$ then we say that $\mu$ is
$n$-Ahlfors-David regular, or $n$-AD regular. A set $E \subset \R^d$ is $n$-AD regular if the $n$-dimensional Hausdorff measure restricted to $E$, denoted by $\mathcal{H}^n\lfloor_E $, is $n$-AD regular.

The space of finite complex Radon measures in $U \subset \R^d$ is denoted by $M(U)$. This is a Banach space with the norm of the total variation:
$\|\nu\|= |\nu|(U)$.

We say that
$k(\cdot,\cdot):\R^d\times\R^d\setminus\{(x,y)\in\R^d\times\R^d:x=y\}\rightarrow\C$ is
an $n$-dimensional  Calder\'{o}n-Zygmund (CZ) kernel if there exist
constants $c>0$ and $\eta$, with $0<\eta\leq1$, such that the
following inequalities hold for all $x,y\in\R^d$, $x\neq y$:
\begin{equation*}
\label{defcz}
\begin{split}
&|k(x,y)|  \leq \frac{c}{|x-y|^n}, \qquad{\mbox{and}} \\ &|k(x,y)
-k(x',y)|  + |k(y,x) -k(y,x')|\leq \frac{c\,|x-x'|^\eta}{|x-y|^{n+\eta}} \quad \mbox{ if
$|x\!-\!x'|\leq |x\!-\!y|/2$.}
\end{split}
\end{equation*}

Given a positive or complex Radon measure $\nu$ on $\R^d$ and a Calder\'on-Zygmund kernel $k$, we define
\begin{equation*}
T^k\nu(x) := \int k(x,y)\,d\nu(y), \qquad{x\in\R^d\setminus\supp\nu}.
\end{equation*}
This integral may not converge when $x\in\supp\nu$. For this reason, we
consider the following $\ve$-truncated operators $T^k_\ve$, $\ve>0$:
$$T^k_\ve\nu(x) := \int_{|x-y|>\ve} k(x,y)\,d\nu(y), \qquad{x\in\R^d}.$$

Given a fixed positive Radon measure $\mu$ on $\R^d$ and $f\in
L^1_{loc}(\mu)$, we write
$$T^k_\mu f(x) := T^k(f\,\mu)(x),\qquad x\in\R^d\setminus\supp(f\,\mu),$$
and
$$T^k_{\mu,\ve} f(x) := T^k_\ve(f\,\mu)(x).$$
If $\nu$ is a positive Radon measures as well we say that $T^k_\nu$ is bounded 
from $L^p(\nu)$ to $L^p(\mu),\ 1<p<\infty$, if the operators $T^k_{\nu,\ve}$ are bounded 
from $L^p(\nu)$ to $L^p(\mu)$ uniformly on $\ve>0$. We also say that the operators $T^k_\nu$ are bounded 
from $L^p(\nu)$ to $L^{p,\infty}(\mu)$ for $1\leq p<\infty$ if for all $f \in L^p(\nu)$ and for all 
 $\lambda>0$,
\begin{equation*}
\mu\{x \in \R^d: |T_\ve(f\nu)(x)|>\lambda\} \leq \frac{c}{\lambda^p}\|f\|^p_{L^p(\nu)},
\end{equation*}
uniformly on $\ve$.
Analogously $T^k$ is bounded from $M(U), \, U \subset \R^d,$ into $L^{1,\infty}(\mu)$ if
there exists some constant $c$ such that for all $\nu\in M(U)$
and all $\lambda>0$,
$$\mu\{x\in\R^d:\,|T^k_\ve\nu(x)|>\lambda\} \leq \frac{c\|\nu\|}\lambda$$
uniformly on $\ve>0$.

Our first result reads as follows.

\vspace{2mm}
\begin{theorem} 
\label{main}
Let $U\subset \R^d$ be a domain with $(d-1)$-AD regular boundary $\Gamma$. 
Let $\mu, \nu$ two measures with $(d-1)$-growth such that $\mu(\R^d\setminus \bar{U})=\nu(U)=0$. 
Let $k$ be a $(d-1)$-dimensional Calder\'on-Zygmund kernel such that the operator $T_{\mathcal{H}^{d-1} \lfloor_\Gamma}^k:L^2(\mathcal{H}^{d-1} \lfloor_\Gamma)\ra L^2(\mathcal{H}^{d-1} \lfloor_\Gamma)$ is bounded. Then, 
\begin{enumerate}
\item the operators $T_\nu^k:L^p(\nu) \rightarrow L^{p}(\mu)$ and $T_\mu^k:L^p(\mu) \rightarrow L^{p}(\nu)$
are bounded for all $1<p<\infty$, 
 
\item $T^k$ is bounded from $M(\R^d \setminus U)$ to $L^{1,\infty}(\mu)$ and from $M(\bar{U})$ to $L^{1,\infty}(\nu)$. In particular, 
the operators $T_\nu^k:L^1(\nu) \rightarrow L^{1,\infty}(\mu)$ and $T_\mu^k:L^1(\mu) \rightarrow L^{1,\infty}(\nu)$
are bounded.

\end{enumerate}
\end{theorem}

\vspace{2mm}

An $n$-$AD$-regular set $E$ is $n$-uniformly rectifiable if there exist $\theta, M >0$ such that for all $x \in E$ and all $r>0$ 
there exists a Lipschitz mapping $\rho$ from the ball $B_n(0,r)$ in $\R^{n}$ to $\R^d$ with $\text{Lip}(\rho) \leq M$ such that$$
\mathcal{H}^n \lfloor_E (B(x,r)\cap \rho(B_{n}(0,r)))\geq \theta r^{n}.$$

Any convolution kernel $k:\mathbb{R}^{d} \setminus \{0\}\rightarrow \mathbb{R}$
such that for all $x\in \mathbb{R}^{d}\setminus \{0\}$,
\begin{equation}
\label{ker}
k(-x)=-k(x)\; \ \text{ and }\ \;\left| \nabla^j k(x)\right| \leq c_j \left| x\right| ^{-n-j}, \ \text{ for}\ j=0,1,2,
\end{equation} defines a bounded operator in $L^2(\mathcal{H}^n \lfloor_E)$ whenever $E$ is an $n$-uniformly rectifiable set. This was originally proved by David, see e.g. \cite{dpaper} and \cite{d88}, under the additional assumption $\left| \nabla^j k(x)\right| \leq c_j \left| x\right| ^{-n-j}$ for all $j\geq 0$. A proof for all kernels satisfying (\ref{ker}) can be found in \cite{tun}. It follows that all such $(d-1)$-dimensional kernels satisfy the hypotheses of Theorem \ref{main}.
 
Relating the $L^2(\mu)$-boundedness of $T^k_\mu$ with the geometric structure of $\mu$ is a hard and largely unresolved problem. After David's result in \cite{d88}, David and Semmes proved a result that goes in the converse direction. In \cite{DS} they showed that the $L^2(\mu)$ boundedness of all operators associated with convolution, odd, $C^\infty$ away from the origin $CZ$ kernels  imply that the measure $\mu$ is $n$-uniformly rectifiable. The David-Semmes conjecture, dating from 1991, asks if the $L^2(\mu)$-boundedness of the operators associated with just one of these kernels, specifically to the $n$-dimensional Riesz kernel $x/|x|^{n+1}$, suffices to imply $n$-uniform rectifiabilty. The conjecture has been very recently resolved in \cite{NToV} in the codimension 1 case, that is for  $n=d-1$. Mattila, Melnikov and Verdera in \cite{MMV} had earlier proved the conjecture in the case of $1$-dimensional Riesz kernels. For all other dimensions and for other kernels few things are known. There are several examples of kernels whose boundedness does not imply rectifiability, see \cite{C}, \cite{Dc} and \cite{hu}. On the other hand in \cite{CMPT} the kernels $\re (z)^{2n-1} / |z|^{2n}, \, z \in \C, n \in \N,$ were considered and it was proved that the $L^2$-boundedness of the operators associated with any of these kernels implies rectifiability. By now, these are the only known examples of convolution kernels not directly related to the Riesz kernels with this property.

With the previous discussion in mind, Theorem \ref{main} elaborates that the boundedness of $T^k_\mu:L^2(\mu)\ra L^2(\nu)$ with $\mu$ and $\nu$ being separated measures as in the theorem holds much more generally than the boundedness of $T^k_\mu$ from $L^2(\mu)$ to $L^2(\mu)$. Notice that in our assumptions $\mu$ and $\nu$ can be any measures with $n$-growth as long as they are separated in a reasonably nice manner. Furthermore we consider general $n$-dimensional $CZ$-kernels requiring less smoothness than in (\ref{ker}). 
 
In \cite{CM} it was shown that for a smaller class of kernels and for $1<p<\infty$ 
the operators $T^k$ are bounded from $L^p(\nu)$ to $L^p(\mu)$ and from $L^p(\mu)$ to $L^p(\nu)$ 
whenever $\mu$ and $\nu$ have $(d-1)$-growth and they are separated by $(d-1)$-Lipschitz graphs. Theorem \ref{main} extends the admissible boundaries from Lipschitz graphs 
to  uniformly rectifiable sets and moreover it covers the endpoint weak-(1,1) case, 
which did not follow from the methods in \cite{CM} and thus it was left untreated there. 

Our proof follows an altogether different approach which makes use of new Calder\'on-Zygmund decompositions partially inspired by the techniques in \cite{tow11}. We should also remark that our proof, as well as the one in \cite{CM}, makes extended use of the following theorem of David from \cite{dpaper}.
\vspace{2mm}
\begin{theorem}
\label{davgen}
Let $\mu, \nu$ two measures with compact support such that $\mu$ is $n$-AD regular and $\nu$ has $n$-growth. Let $k$ be an $n$-dimensional Calder\'on-Zygmund kernel 
such that $T^k:L^2(\mu)\rightarrow L^2(\mu)$ is bounded.
 Then  the operators $T_\mu^k:L^p(\nu) \rightarrow L^{p}(\mu)$ and  $T_\nu^k:L^p(\mu) \rightarrow L^{p}(\nu)$
are bounded for all $1< p <\infty$, 
\end{theorem}
\vspace{2mm}

Using an example based on the four-corners Cantor set we prove that there exist $1$-growth measures $\nu$ and $\mu$, $\mu$ not $AD$-regular, such that the Cauchy singular integral operator, which is associated with the Cauchy kernel $1/z$, is bounded in $L^2(\mu)$ but not from $L^2(\mu)$ to  $L^2(\nu)$.  Hence we show that Theorem \ref{davgen} fails without the  $AD$-regularity assumption on $\mu$. 

On the other hand the use of Calder\'on-Zygmund decompositions can be exploited even further as only minor modifications in the proof of Theorem \ref{main} allow us to prove the following endpoint result which, as far as we know, is new.

\vspace{2mm}
\begin{theorem}
\label{davgenw}
Under the assumptions of Theorem \ref{davgen} the operator  $T^k$ is bounded  from $M(\supp \mu)$ to  $L^{1,\infty}(\nu)$. In particular $T_\mu^k:L^1(\mu) \rightarrow L^{1,\infty}(\nu)$ is bounded.
\end{theorem}
\vspace{2mm}

Let us remark that the boundedness of the operator
$T^k:M(\supp \nu)\rightarrow L^{1,\infty}(\nu)$ also holds. This is due to the fact that the boundedness
of $T^k_\mu$ in $L^2(\mu)$ implies the boundedness from $M(\R^d)$ to $L^{1,\infty}(\mu)$ (see
\cite{tow11} or \cite[Chapter 2]{Tolsa-book}, for example).

The paper is organised as follows. In Section 2 we prove the appropriate Calder\'on-Zygmund decompositions needed for the proof of Theorem \ref{main} and in Section 3 we prove Theorem \ref{main}. The proof of Theorem \ref{davgenw} is outlined in Section 4. Finally in Section 5 we prove that the $AD$-regularity assumption is essential for the proof of Theorem \ref{davgen}.

Throughout the paper the letter $C$ stands
for some constant which may change its value at different
occurrences. The notation $A\lesssim B$ means that
there is some fixed constant $C$ such that $A\leq CB$,
with $C$ as above. Also, $A\approx B$ is equivalent to $A\lesssim B\lesssim A$.

\section{Calder\'on-Zygmund Decompositions}\label{sec2}

For any set $A\subset \R^d$ and $\ve>0$ let $N(A,\ve)=\{x \in \R^d:\dist(x,A)\leq \ve\}$. 

\begin{theorem}\label{teodescz1}
\label{czd2}
Let $U\subset \R^d$ be a domain with $(d-1)$-AD regular boundary $\Gamma$.  
Let $\mu, \nu$ two measures with $(d-1)$-growth and compact support such that $\mu(\R^d\setminus \bar{U})=\nu(U)=0$.  Suppose that $\supp\nu\subset N(\Gamma,\diam(\Gamma))$.
Then for all $f \in L^p(\nu), \ 1\leq p <\infty,$ and for all 
 $\lambda>\bigl(2^{d+1}\,\|f\|^p_{L^p(\nu)}/\|\mu\|\bigr)^{1/p}$:
\begin{itemize}
\item[(a)] There exists a family of
almost disjoint balls $\{B_i\}_i$ (that is, $\sum_i \chi_{B_i} \leq
c$) centered at $\supp \nu$, with radius not exceeding $3\,\diam(\Gamma)$, and a function $h\in L^1(\mathcal{H}^{d-1}\lfloor _\Gamma)$ such that
\begin{equation}  \label{cc1}
 \int_{B_i}|f|^p d\nu >\frac{\lambda^p}{2^{d+1}}\,\mu(2B_i),
\end{equation}
\begin{equation}  \label{cc2}
\int_{\eta B_i}|f|^p d\nu \leq
\frac{\lambda^p}{2^{d+1}}\,\mu(2\eta B_i) \quad
\mbox{for $\eta >2$,}
\end{equation}
\begin{equation}  \label{cc3}
f \nu \lfloor_{\R^d\setminus\bigcup_i
B_i} = h\,\hd \lfloor_\Gamma  \quad\mbox{with $|h|\leq c \lambda$ \; $\mathcal{H}^{d-1}$ a.e. in $\Gamma$ .}
\end{equation}

\vv
\item[(b)] For each $i$, let $R_i$ be a ball concentric
with $B_i$, with $10 r(B_i)\leq r(R_i) \leq 30 \diam(\Gamma)$ and denote
$w_i= \frac{\chi_{B_i}}{\sum_k \chi_{B_k}}$. Then,
there exists a family of functions
$\vphi_i$ with $\supp(\vphi_i)\subset R_i \cap \Gamma$ and with constant sign satisfying
\begin{equation}  \label{cc4}
\int_\Gamma \vphi_i \,d\hd = \int_{B_i} w_i f\,d\nu,
\end{equation}
\begin{equation}  \label{cc5}
\sum_i |\vphi_i| \leq c_1\,\lambda
\end{equation}
(where $c_1$ is some fixed constant), and
\begin{equation}  \label{cc6}
\|\vphi_i\|_{L^\infty(\hdg)} \,r(R_i)^{d-1}\leq
c\, \int_{B_i}|f|d \nu.
\end{equation}

\end{itemize}
\end{theorem}

\begin{proof}
{\bf(a)} 
Let
$$F=\left\{x \in \supp \nu: \text{ there exists } B_x, \text{centered at}\ x \text{ such that }\int_{B_x}|f|^p d \nu>\frac{\lambda^p}{2^{d+1}} \mu(2 B_x)\right\},$$
and $G=\supp \nu \setminus F$. Notice that $\supp \nu \setminus \Gamma \subset F$ and $G \subset \Gamma$.

For all $x \in F$ let $B_x$ be a maximal ball centered in $x$ in the sense that
$$\int_{B_x}|f|^p d \nu >\frac{\lambda^p}{2^{d+1}}\,\mu(2B_x),$$
but for all concentric balls $D_x$ with $r(D_x) > 2 r(B_x)$
$$\int_{D_x}|f|^p d \nu \leq\frac{\lambda^p}{2^{d+1}}\,\mu(2D_x).$$
Notice that this maximal ball exists. Indeed, if $B_x'$ is centered at $x$, contains 
$\supp \mu \cup \supp \nu$, and satisfies \eqref{cc1}, we have
$$\int |f|^p d \nu=\int_{B_x'}|f|^pd \nu>\frac{\lambda^p}{2^{d+1}}\mu(2B_x')=\frac{\lambda^p}{2^{d+1}}\|\mu\|$$
which contradicts the initial assumption. Notice also that since $\supp \nu \subset N(\Gamma,\diam(\Gamma))$ all the maximal balls $B_x$ satisfy $r(B_x) \leq 3 \diam(\Gamma)$. 

Applying Besicovitch's covering theorem we get an almost disjoint subfamily of balls $\{B_i\}_i \subset \{B_x\}_x$ which covers $F$ and satisfy \rf{cc1} and \rf{cc2}
by construction.

Let $\tau=|f|^p \nu$. Recall that, given $\alpha>1$ and $\beta >\alpha^n$, a ball $B(x,r)$ is called $\tau$-$(\alpha,\beta)$-doubling  if $\tau(\alpha B(x,r)) \leq \beta \tau(B(x,r))$. Denote by $D$ the set of points $z \in \supp \tau$ such that there exists a sequence of $\tau$-$(2,2^{d+1})$-doubling balls $P^z_k$ centered at $z$ such that $r(P^z_k)\ra 0$. By standard arguments it follows that $\tau(G \cap D)=\tau(D)$. Therefore for $\tau$-a.e. $z \in G$, there exists a sequence of 
$(2,2^{d+1})$-$\tau$-doubling balls $P_k$ centered at $z$, with $r(P_k)\to0$, such that
$$\tau(P_k) \leq\frac{\lambda^p}{2^{d+1}}\,\mu(2P_k),$$
and thus
$$\tau(2P_k)\leq 2^{d+1}\tau(P_k) \leq\lambda^p\,\mu(2P_k).$$
This implies that $\tau \lfloor_G$ is absolutely continuous with respect to
$\mu$ and that $\tau \lfloor_G= h_1\mu$ with $|h_1|\leq\lambda^p$ $\mu$-a.e., by
the Lebesgue-Radon-Nikodym theorem (see \cite[p. 36-39]{mb}, for instance). 

Notice that  if $A\subset U$ then $h_1 \mu(A)=\tau(A\cap G)=0$ and if $A \subset \R^d \setminus \bar{U}$ then $\tau(G \cap A)=h_1\mu(A)=0$. Therefore
\begin{equation}
\label{nuhau}
\tau \lfloor_G= h_1 \mu \lfloor_\Gamma\ \;\text{ with } \ \;0\leq h_1\leq \lambda^p\ \;\mu-\text{a.e. in }\Gamma.
\end{equation}

Since $\mu \lfloor_\Gamma$ is supported on $\Gamma$ and it has $(d-1)$-growth by standard differentiation theory of measures, see e.g.\cite{mb},
it is absolutely continuous
with respect to $\hdg$ with bounded
Radon-Nikodym derivative. In other words, there exists a Borel function $h_2$
such that  
\begin{equation}
\label{m1hau}
\mu \lfloor_ \Gamma=h_2 \hdg\ \;\text{ and } \ \;0\leq h_2\leq c \ \; \hdg-\text{a.e.}.
\end{equation}
By (\ref{nuhau}) and (\ref{m1hau}) we deduce that
\begin{equation}
\label{nuhaua}
\tau \lfloor_G= h_3\hdg
\end{equation}
where $h_3=h_1\,h_2$ and $|h_3| \leq c \lambda^p$,  $\mathcal{H}^{d-1}$-a.e. in $\Gamma$. 

Now for any ball $B$ centered in $G$, using H\"older's inequality and (\ref{nuhaua}),
\begin{equation*}
\begin{split}
f \nu \lfloor_{ G} (B)&=\int_{B \cap G}f d\nu \leq \left( \int_{B \cap G}  |f|^p d \nu\right)^{1/p} \nu(B)^{1/p'}\\
&\lesssim \tau(B\cap G)^{1/p} r(B)^{\frac{d-1}{p'}}\\
&\lesssim \left(\int_{B \cap G} h_3 d \hd \right)^{1/p} \hd (\Gamma \cap B)^{1/p'}\\
&\lesssim \left(\lambda^p \hd (\Gamma \cap B) \right)^{1/p} \hd (\Gamma \cap B)^{1/p'} =\lambda \hd (\Gamma \cap B).
\end{split}
\end{equation*}
Therefore $\nu \lfloor_G$ is absolutely continuous with respect to $\hdg$ and
\begin{equation}
\label{nuhaua2}
f \nu \lfloor_G= h\,\hdg,
\end{equation}
where $0\leq h \leq c \lambda$,  $\mu$-a.e..

\vspace{3mm}
\noindent{\bf(b)}  Assume first that the family of balls $\{B_i\}_i$ is finite.
Then we may suppose that this family is ordered in such a way
that the sizes of the balls $R_i$ are non decreasing (i.e. $\ell(R_{i+1})\geq
\ell(R_i)$).
The functions $\vphi_i$ that we will construct will be of the form $\vphi_i
=\alpha_i\,\chi_{A_i}$, with $\alpha_i\in\R$ and $A_i\subset R_i$.
We set $A_1=R_1$ and
$\vphi_1 = \alpha_1\,\chi_{R_1},$
where the constant $\alpha_1$ is chosen so that $\int_{B_1}w_1 f\,d\nu=\int_\Gamma
\vphi_1\,d\hd$.

Suppose that $\vphi_1,\ldots,\vphi_{k-1}$ have been constructed,
satisfy \rf{cc4} and $$\sum_{i=1}^{k-1} |\vphi_i|\leq
c_1\,\lambda,$$  where $c_1$ is some constant which will be fixed below.
 Let $R_{s_1},\ldots,R_{s_m}$ be the subfamily of
$R_1,\ldots,R_{k-1}$ such that $R_{s_j}\cap R_k \neq \varnothing$. As $\ell(R_{s_j}) \leq \ell(R_k)$ (because of the non decreasing sizes of $R_i$),
we have $R_{s_j} \subset 3R_k$. 

Since the $B_i$'s are maximal we have that
$$\frac{1}{\mu(6 B_i)} \int_{3B_i} |f|^p d \nu\leq \frac{\lambda^p}{2^{d+1}}.$$
Hence $\supp \mu \cap 6B_i \neq \emptyset$ and thus $\Gamma \cap 6B_i \neq \emptyset$ as well. Choosing any $z_i \in \Gamma \cap 6B_i$, $B(z_i,d_i) \subset R_i$ for $d_i=r(R_i)-6r(B_i)$. Since $r(R_i)-\frac{6}{10}r(R_i)\leq d_i \leq 30 \diam(\Gamma)$ and $\Gamma$ is $(d-1)$-AD-regular we deduce that
\begin{equation}
\label{adre}
\hd(\Gamma \cap R_i) \geq \hd(\Gamma \cap B(z_i,d_i)) \geq C r(R_i).
\end{equation} 
Now taking into account that for $i=1,\ldots,k-1$, by \rf{cc4},
$$\int_\Gamma |\vphi_i|\,d\hd \leq \int_{B_i}|f| d \nu,$$
and using the finite ovelarpping of the balls $B_{s_j}$, H\"older's inequality, the $(d-1)$-growth of $\mu$ and $\nu$ and (\ref{adre}), it follows that
\begin{equation*}
\begin{split}
\sum_j \int_\Gamma |\vphi_{s_j}|\,d\hd &\leq  \sum_j
\int_{B_{s_j}}|f| d \nu \lesssim \int_{3R_k} |f| d \nu \\
&\leq \left( \int_{3 R_k} |f|^p d \nu\right)^{1/p} \nu (3 R_k)^{1/p'}\\
&\leq c\lambda \mu(6R_k)^{1/p}\nu (3 R_k)^{1/p'}\\
&\leq c \lambda r(R_k)^{d-1}\leq \,c_2\lambda\,\hd(R_k \cap \Gamma).
\end{split}
\end{equation*}
Therefore, by Chebyshev,
$$\hd\left\{\Gamma \cap \{{\textstyle \sum_j} |\vphi_{s_j}| > 2c_2\lambda\}\right\}\leq \frac{1}{2 c_2 \lambda}\int_{\Gamma}{\textstyle \sum_j} |\vphi_{s_j}| d \hd \leq \frac{\hd(R_k \cap \Gamma)}{2}.$$ Setting
$$A_k = \Gamma \cap R_k\cap\left\{{\textstyle \sum_j} |\vphi_{s_j}| \leq
2c_2\lambda\right\},$$
we have $\hd(A_k)\geq \hd(R_k\cap \Gamma)/2.$

The constant $\alpha_k$ is chosen so that for $\vphi_k = \alpha_k\, \chi_{A_k}$
we have $\int_\Gamma \vphi_k\,d\hd = \int_{B_k} w_k f\,
d\nu$.  Then, using also (\ref{adre}), we obtain
\begin{equation}
\begin{split}
\label{akest} 
|\alpha_k|  &\leq  \frac{\int_{B_k}|f|d \nu}{\hd(A_k)}
\leq \frac{2\int_{\frac{1}{2}R_k}|f|d \nu}{\hd(R_k \cap \Gamma)}\\
&\leq c \,\frac{\left(\int_{\frac{1}{2}R_k} |f|^p d \nu \right)^{1/p} \nu(R_k)^{1/p'}}{r(R_k)^{d-1}} \lesssim \lambda \,\frac{\mu(R_k)^{1/p}\nu(R_k)^{1/p'}}{r(R_k)^{d-1}} \leq c_3\lambda.
\end{split}
\end{equation}
Thus in $A_k$,
$$|\vphi_k|+\sum_{j} |\vphi_{j}|=|\vphi_k|+\sum_{j} |\vphi_{s_j}| \leq (2c_2+c_3)\lambda.$$
Since $\sum_{j=1}^{k-1} |\vphi_{j}(x)| \leq c_1 \lambda$ for $x \notin A_k$ by the previous steps of the induction, if we choose $c_1=2c_2+c_3$, \rf{cc5} follows.

Now it is easy to check that \rf{cc6} also holds. Indeed we have
\begin{align*}
\|\vphi_i\|_{L^\infty(\hdg)}\, r(R_i)^{d-1} &
 \leq c\,|\alpha_i|\,\hd( \Gamma \cap R_i)\leq c\,|\alpha_i|\,\hd( \Gamma \cap A_i)\\ 
& =  c\,\left|\int_{B_i}w_i f\,d\nu\right| 
\, \leq \, c\, \int_{B_i}|f|d\nu.
\end{align*}

Suppose now that the collection of balls $\{B_i\}_i$ is not finite.
For each fixed $N$ we consider the family of balls $\{B_i\}_{1\leq i \leq N}$.
Then, as above, we construct functions $\vphi_1^N,\ldots,\vphi_N^N$ with
$\supp(\vphi_i^N)\subset R_i$ satisfying
$$\int_\Gamma \vphi_i^N \,d\hd = \int_{B_i} w_i f\,d\nu,$$
$$\sum_{i=1}^N |\vphi_i^N| \leq B\,\lambda,
$$
and
$$\|\vphi_i^N\|_{L^\infty(\hdg)}\, r(R_i)^{d-1} \leq c\, \int_{B_i}|f|d\nu.$$
Notice that the sign of $\vphi_i^N$ equals the sign of $\int w_i f\,d\nu$ and
so it does not depend on $N$.

Then there is a subsequence
$\{\vphi_1^k\}_{k\in I_1}$ which is convergent in the weak $\ast$ topology of
$L^\infty(\hdg)$ to some function
$\vphi_1\in L^\infty(\hdg)$. Now we can consider a subsequence
$\{\vphi_2^k\}_{k\in I_2}$ with $I_2\subset I_1$ which
is also convergent in the weak $\ast$ topology of $L^\infty(\hdg)$ to some 
function $\vphi_2\in L^\infty(\hdg)$.
In general, for each $j$ we consider a subsequence
$\{\vphi_j^k\}_{k\in I_j}$ with $I_j\subset I_{j-1}$ that converges
in the weak $\ast$ topology of $L^\infty(\hdg)$ to some function
$\vphi_j\in L^\infty(\hdg)$. It is easily checked that the functions
$\vphi_j$ satisfy the required properties.
\end{proof}

For a domain $U$, $\Gamma$ and $\mu$ as in Theorem \ref{czd2}, and a complex measure
 $\nu \in M(\R^d \setminus U)$ we have the following result analogous to the preceding one.

\begin{theorem}
\label{czd1}Let $U\subset \R^d$ be a domain with $(d-1)$-AD regular boundary $\Gamma$. 
Let $\mu$ be a measure with $(d-1)$-growth and compact support such that $\mu(\R^d\setminus \bar{U})=0$. 
Then for all $\nu \in M(\R^d \setminus U)$ such that $\supp \nu \subset N(\Gamma,\diam(\Gamma))$ and for all 
 $\lambda>2^{d+1}\,\|\nu\|/\|\mu\|$:
\begin{itemize}
\item[(a)] There exists a family of
almost disjoint balls $\{B_i\}_i$ (that is, $\sum_i \chi_{B_i} \leq
c$) centered at $\supp \nu$,  with radius not exceeding $3\,\diam(\Gamma)$,  and a function $g\in L^1(\mathcal{H}^{d-1}\lfloor _\Gamma)$ such that
\begin{equation}  \label{cc1'}
 |\nu|(B_i) >\frac{\lambda}{2^{d+1}}\,\mu(2B_i),
\end{equation}
\begin{equation}  \label{cc2'}
 |\nu|(\eta B_i) \leq
\frac{\lambda}{2^{d+1}}\,\mu(2\eta B_i) \quad
\mbox{for $\eta >2$,}
\end{equation}
\begin{equation}  \label{cc3'}
\nu \lfloor_\Gamma = g\,\hd \lfloor_{\Gamma\setminus\bigcup_i
B_i}  \quad\mbox{with $|g|\leq c \lambda$ \; $\mathcal{H}^{d-1}$ a.e. in $\Gamma$ .}
\end{equation}

\vv
\item[(b)] For each $i$, let $R_i$ be a ball concentric
with $B_i$, with $10 r(B_i)\leq r(R_i) \leq 30 \diam(\Gamma)$ and denote
$w_i= \frac{\chi_{B_i}}{\sum_k \chi_{B_k}}$. Then,
there exists a family of functions
$\vphi_i$ with $\supp(\vphi_i)\subset R_i \cap \Gamma$ and with constant sign satisfying
\begin{equation}  \label{cc4'}
\int_\Gamma \vphi_i \,d\hd = \int_{B_i} w_i\,d\nu,
\end{equation}
\begin{equation}  \label{cc5'}
\sum_i |\vphi_i| \leq c_1\,\lambda
\end{equation}
(where $c_1$ is some fixed constant), and
\begin{equation}  \label{cc6'}
\|\vphi_i\|_{L^\infty(\hdg)} \,r(R_i)^{d-1}\leq
c\, |\nu|(B_i).
\end{equation}

\end{itemize}
\end{theorem}

This result can be derived from Theorem \ref{teodescz1} by setting $p=1$, taking $f$ such that
$\nu=f\,|\nu|$, and replacing the measure $\nu$ there by $|\nu|$.

\section{Weak ($p,p$) boundedness}

We will split the proof of Theorem \ref{main} into two parts. We present first the proof of the boundedness of $T^k$ from the space of measures  $M(\R^d \setminus U)$ into
$L^{1,\infty}(\mu)$. Later we will show that $T_\nu^k$ is bounded from $L^p(\nu)$ to $L^{p,\infty}(\mu)$
for $p>1$, by  similar (although more and technical) arguments. By Theorem \ref{davgen} and interpolation it then follows that $T_
\nu^k$ is bounded from $L^p(\nu)$ to $L^{p}(\mu)$.

\vspace{2mm}
\begin{theorem}
\label{w11}
Let $U\subset \R^d$ be a domain with $(d-1)$-AD regular boundary $\Gamma$. 
 Let $\mu$ be a measure with $(d-1)$-growth such that $\mu(\R^d\setminus \bar{U})=0$.
Let $k$ be a $(d-1)$-dimensional Calder\'on-Zygmund kernel such that the operator $T_{\mathcal{H}^{d-1} \lfloor_\Gamma}^k:L^2(\mathcal{H}^{d-1} \lfloor_\Gamma)\ra L^2(\mu)$ is bounded.
Then the operator $T^k$ is bounded from $M(\R^d \setminus U)$ into
$L^{1,\infty}(\mu)$. That is for all $\nu \in M(\R^d \setminus U)$ and for all 
 $\lambda>0$,
\begin{equation}
\label{eqdebil49}
\mu(\{x \in \R^d: |T_\ve^k\nu(x)|>\lambda\}) \leq \frac{c}{\lambda}\|\nu\|,
\end{equation}
uniformly on $\ve$.
\end{theorem}

To simplify notation, below we will write $T$ instead of $T^k$.

\begin{proof}
Suppose first that both $\mu$ and $\nu$ have compact support and $\supp \nu \subset N(\Gamma, \diam(\Gamma))$. 
Clearly, we may assume that $\lambda>2^{d+1}\|\nu\|/\|\mu\|$.

Let $\{B_i\}_i$ be the almost disjoint
family of balls of Theorem \ref{czd1}.
Let $R_i=10 B_i$ and notice that $r(R_i) \leq 30 \diam(\Gamma)$ recalling that $r(B_i)\leq 3 \diam(\Gamma)$. Then we can write $\nu=\kappa+\beta$, with
$$\kappa= \nu\lfloor_{\R^d\setminus\bigcup_i B_i} + \sum_i \vphi_i\hdg$$
and
$$\beta = \sum_i \beta_i := \sum_i \left(w_i\, \nu \lfloor_{B_i} - \vphi_i\hdg\right),$$
where the functions $\vphi_i$ satisfy \rf{cc4'}, \rf{cc5'} \rf{cc6'} and
$w_i= \frac{\chi_{B_i}}{\sum_k \chi_{B_k}}$.
Moreover, $\nu \lfloor_{\R^d \setminus \cup B_i}= g \hdg$ with $|g| \leq c \lambda$ $\hd$-a.e. in $\Gamma$
by \rf{cc3'}. Therefore $\kappa= \tilde{g} \hdg$ with $\tilde{g}=\sum_i \vphi_i + g.$ In particular, $|\tilde{g}| \leq C \lambda$ $\hd$-a.e. in $\Gamma$.

By \rf{cc1} we have
$$\mu\Bigl(\bigcup_i 2B_i\Bigr) \leq \frac{c}{\lambda} \sum_i
|\nu|(B_i)\leq \frac{c}{\lambda}\, \|\nu\|.$$
So we have to prove that
\begin{equation}  
\label{eqq}
\mu\Bigl\{x\in \R^d\setminus \bigcup_i 2B_i:\,|T_{\ve} \nu(x)|>\lambda\Bigr\} \leq
\frac{c}{\lambda}\,\|\nu\|.
\end{equation}

We will first show that
\begin{equation}
\label{fint}
\int_{\R^d\setminus\bigcup_k 2B_i} |T_{\ve} \beta|\,d\mu \leq C \|\nu\|.
\end{equation}
Since $\beta_i(R_i)=\int_{B_i} w_i d \nu_i - \int_\Gamma \vphi_i d  \hd=0$ and $\supp(\beta_i)\subset R_i$ by  standard estimates 
we deduce that
$$\int_{\R^d\setminus 2R_i} |T_{\ve} \beta_i|\,d\mu \leq C\,\|\beta_i\| \leq C(|\nu|(B_i)+\int_\Gamma |\vphi_i| d \hd )\leq c|\nu|(B_i)
.$$
We will now check that
\begin{equation}  \label{27b}
\int_{2R_i\setminus 2B_i} |T_{\ve} \beta_i|\,d\mu \leq
c\,|\nu|(B_i).
\end{equation}
Observe that $|T_\ve(w_i \nu \lfloor_{B_i})(x)|\leq c\,|\nu|(B_i)/r(B_i)$ for any $x\in 2R_i\setminus B_i$.
Therefore,
\begin{equation}\label{wivi}
\begin{split}
\int_{2R_i \setminus 2 B_i}|T_\ve(w_i \nu \lfloor_{B_i})|d \mu \leq 
 C  \frac{|\nu| (B_i)\,\mu(2 R_i)}{r(B_i)^{d-1}} \leq C |\nu| (B_i),
\end{split}
\end{equation}
because $\mu$ has $(d-1)$-growth and $R_i=10 B_i$.

On the other hand, using Cauchy-Schwarz and (\ref{cc6}) we get
\begin{equation}
\begin{split}
\label{fi}
\int_{2R_i} |T_{\ve}( \vphi_i \hdg)| \,d\mu & \leq 
\left(\int |T_{\hdg,\ve}( \vphi_i)|^2 \,d\mu\right)^{1/2} \, \mu(2R_i)^{1/2} \\
& \leq  c
\left(\int_\Gamma |\vphi_i|^2 \,d\hd\right)^{1/2} \, \mu(2R_i)^{1/2} \\
& \leq c \left( \|\vphi_i\|^2_{L^\infty(\hdg)} \,\hd(\Gamma \cap R_i)\right)^{1/2}\mu(2R_i)^{1/2} \\
&\leq c\,\|\vphi_i\|_{L^\infty(\hdg)} r(R_i)^{d-1} \leq c\,|\nu|(B_i).
\end{split}
\end{equation}

Combining (\ref{wivi}) and (\ref{fi}), we obtain (\ref{27b}). Then we deduce
\begin{align*}
\int_{\R^d\setminus\bigcup_k 2B_k} |T_{\ve} \beta|\,d\mu & \leq 
\sum_i \int_{\R^d\setminus\bigcup_k 2B_k} |T_{\ve} \beta_i|\,d\mu  \leq  c\,\sum_i |\nu|(B_i) \, \leq \, c\,\|\nu\|,
\end{align*}
by the finite overlap of the balls $B_i$. Thus (\ref{fint}) is proven. This implies that
\begin{equation} \label{bb}
\mu\Bigl\{x\in \R^d\setminus \bigcup_i 2B_i:\,|T_{\ve} \beta(x)|>\lambda/2\Bigr\} \leq
\frac{c}{\lambda}\|\nu\|.
\end{equation}

Recalling that $\kappa= \tilde{g} \hdg= \nu\lfloor_{\R^d\setminus\bigcup_i B_i} + \sum_i \vphi_i\hdg$ we get
\begin{align*}
\int_\Gamma |\tilde{g}|\,d\hd & \leq  |\nu|(\R^d\setminus \bigcup_i B_i) +
\sum_i \int_\Gamma |\vphi_i|\,d\hd \\
& \leq  \|\nu\| + \sum_i |\nu|(B_i) \, \leq \,
 c\,\|\nu\|,
\end{align*}
by the finite overlap of the balls $B_i$. Taking into account that $|\tilde{g}|\leq c\,\lambda$ $\hd$-a.e. in $\Gamma$ we get
\begin{equation}
\begin{split}  
\label{gg}
\mu\Bigl\{x\in \R^d\setminus & \bigcup_i 2B_i:\,|T_{\ve}\kappa(x)|>\lambda/2\Bigr\} \\ 
&\leq \frac{c}{\lambda^2}\int |T_{\ve}\kappa|^2 d \mu = \frac{c}{\lambda^2} \int |T_{\hdg,\ve}(\tilde{g})|^2 d \mu \\
&\leq \frac{c}{\lambda^2} \int_\Gamma |\tilde{g}|^2\,d\hd \leq \frac{c}{\lambda} \int_\Gamma |\tilde{g}|\,d\hd\leq \frac{c}{\lambda} \|\nu\|.
\end{split}
\end{equation}
Now, by \rf{bb} and \rf{gg} we get \rf{eqq}.

\vspace{2mm}
In the case that $\nu, \mu$ have compact support but $\supp\nu \not\subset N(\Gamma, \diam(\Gamma))$, we split $$\nu=\nu_1+\nu_2:= \nu \lfloor_{N(\Gamma, \diam(\Gamma))}+\nu \lfloor_{\R^d \setminus N(\Gamma, \diam(\Gamma))}.$$
For $\nu_1$ we have shown that the estimate (\ref{eqdebil49}) holds. For $\nu_2$, using that $\dist(\supp \nu_2,\mu)\geq \diam(\Gamma)$ and that $\|\mu\| \leq c_\mu (\diam \Gamma)^{d-1}$, we deduce that
$$|T_{\ve}\nu_2(x)|\leq C\,\frac{\|\nu_2\|}{\diam(\Gamma)^{d-1}}\qquad\mbox{for all $x\in\supp\mu$.}$$
Therefore,
\begin{equation*}
\begin{split}  
\mu\Bigl\{x:\,|T_{\ve}(\nu_2)(x)|>\lambda \Bigr\} &\leq \frac{c}{\lambda}\int |T_{\ve}(\nu_2)|d \mu  
\leq C\,\frac{\|\nu_2\|\,\|\mu\|}{\diam(\Gamma)^{d-1}}\leq C\,\|\nu_2\|.
\end{split}
\end{equation*}

\vspace{2mm}
Suppose now that $\mu$ is compactly supported but not $\nu$. 
Let $N_0$ be such that $\supp\mu\subset B(0,N_0)$, and for some $N>N_0$, let 
$\nu_N = \chi_{B(0,N)}\,\nu$.
 Then, for $x\in\supp(\mu)$,
$$|T_\ve(\nu-\nu_N)(x)|\leq c\,\frac{|\nu|(\R^d\setminus B(0,N))}{(N-N_0)^{d-1}}.$$
Thus $T_\ve\nu_N(x)\to T_\ve\nu(x)$ for all $x\in \supp(\mu)$, and since the estimate 
\rf{eqdebil49} holds
for $\nu_N$, letting $N\to\infty$, we deduce that it also holds for $\nu$.

On the other hand, if $\mu$ is not compactly supported, then for $\mu_N=\mu\lfloor B(0,N)$,
$$\mu_N\{x\in\R^d:\,|T_\ve\nu(x)|>\lambda\}\leq c\frac{\|\nu\|}\lambda$$
uniformly on $N$, and then \rf{eqdebil49} follows in full generality.
\end{proof}
\vspace{2mm}

Thus we have proved (ii) of Theorem \ref{main}. For the proof of (i) we need to introduce some additional notation and recall some well known results. 
If $\mu$ is any non-negative Radon measure, we define a radial maximal function by
$$M_R\mu(x)=\sup_{r>0}r^{1-d}\mu(B(x,r)).$$
If $f$ is a measurable function we also set
$$M^\mu_Rf(x):=M_R(|f|\mu)(x)=\sup_{r>0}r^{1-d}\int_{B(x,r)}|f|d\mu.$$

It follows (see for example \cite{db}) that if $\mu$ and $\nu$ have $(d-1)$-growth and $1 <p <\infty$ then for $f \in L^p(\mu)$,
\begin{equation}
\label{boundmax}
\|M^\mu_Rf\|_{L^p(\nu)} \leq c(p,\mu,\nu) \|f\|_{L^p(\mu)}.
\end{equation}

We now define the $q$-radial Maximal operator for a measurable function $f$ with respect to a non-negative Radon measure $\mu$ by
$$M^\mu_{R,q}(f)(x)=\sup_{r>0}\left( r^{1-d} \int_{B(x,r)}|f|^q d \mu\right)^{1/q}.$$
For $p>q$, noticing that $|g|^q \in L^{p/q}(\mu)$ as $g \in L^{p}(\mu)$ and using (\ref{boundmax})
\begin{equation}
\begin{split}
\label{maxqb}
\|M^\mu_{R,q}(g)\|_{L^{p}(\nu)}^{p}&=\int (M^\mu_{R,q}(g))^{p} d \nu=\int (M^\mu_R(|g|^p))^{p/q}d \nu \\
&\lesssim \int (|g|^q)^{p/q}d \mu=\|g\|_{L^{p}(\mu)}^{p}.
\end{split}
\end{equation}

The following easy lemma can be found for example in \cite{tow11}. 
\begin{lemma} 
\label{facil}
Let $\mu$ be a positive measure on $\R^d$, $x\in\R^d$, and $\rho,\eta>0$, such that
$$\mu(B(x,r)) \leq c_0 r^n$$
for all $r\geq \rho$. Then,
\begin{equation}\label{eqfacil}
\int_{|y-x|\geq \rho} \frac1{|y-x|^{n+\eta}} \,d\mu(y) \leq c(n,\eta)\,\frac{c_0}{\rho^\eta}.
\end{equation}
\end{lemma}

\vspace{2mm}
The statement in (i) follows from the next theorem
and interpolation.
\vspace{2mm}

\begin{theorem}\label{teo32}
\label{wpp}
Let $U\subset \R^d$ be a domain with $(d-1)$-AD regular boundary $\Gamma$. 
Let $\mu, \nu$ be two measures with $(d-1)$-growth such that $\mu(\R^d\setminus U)=\nu(U)=0$.
Let $k$ be a $(d-1)$-dimensional Calder\'on-Zygmund kernel such that the operator $T_{\mathcal{H}^{d-1} \lfloor_\Gamma}^k:L^q(\mathcal{H}^{d-1} \lfloor_\Gamma)\ra L^q(\mu)$ is bounded for some $1<q<\infty$.
 Then the operator $T^k$ is bounded from $L^p(\nu)$ into
$L^{p,\infty}(\mu)$ for $1 < p <q$. That is for all $f \in L^p(\nu)$ and for all 
 $\lambda>0$,
\begin{equation}
\label{eqdebilp}
\mu\{x \in \R^d: |T_\ve^k(f\nu)(x)|>\lambda\} \leq \frac{c}{\lambda^p}\|f\|^p_{L^p(\nu)},
\end{equation}
uniformly on $\ve$.
\end{theorem}

\begin{proof}
Suppose first that both $\mu$ and $\nu$ have compact support and $\supp \nu \subset N(\Gamma, \diam(\Gamma))$. 
Clearly, we may assume that $\lambda^p>2^{d+1}\|f\|^p_{L^p(\nu)}/\|\mu\|$.

Let $\{B_i\}_i$ be the almost disjoint family of balls of Lemma \ref{czd2}.
Let $R_i=10 B_i$ and notice that $r(R_i) \leq 30 \diam(\Gamma)$, since $r(B_i)\leq 3 \diam(\Gamma)$. Then we write $f\nu=\kappa+\beta$, with
$$\kappa= f\nu\lfloor_{\R^d\setminus\bigcup_i B_i} + \sum_i \vphi_i\hdg$$
and
$$\beta = \sum_i \beta_i := \sum_i \left(w_if\, \nu \lfloor_{B_i} - \vphi_i\hdg\right),$$
where the functions $\vphi_i$ satisfy \rf{cc4}, \rf{cc5} \rf{cc6} and
$w_i= \frac{\chi_{B_i}}{\sum_k \chi_{B_k}}$.

By Theorem \ref{czd2}, $f\nu \lfloor_{\R^d \setminus \cup B_i}= h \hdg$ with $|h| \leq c \lambda$ $\hd$-a.e. in $\Gamma$. Therefore $\kappa= \tilde{h} \hdg$ with $\tilde{h}=\sum_i \vphi_i + h$  and $|\tilde{h}| \leq C \lambda$ $\hd$-a.e. in $\Gamma$.
By \rf{cc1} we have
$$\mu\Bigl(\bigcup_i 2B_i\Bigr) \leq \sum_i \mu(2 B_i)\leq\frac{c}{\lambda^p} \sum_i
\int_{B_i}|f|^p d \nu \leq\frac{c}{\lambda^p} 
\int |f|^p d\nu.$$
So it remains to prove that
\begin{equation*}  
\mu\Bigl\{x\in \R^d\setminus \bigcup_i 2B_i:\,|T_{\ve}(f \nu)(x)|>\lambda\Bigr\} \leq
\frac{c}{\lambda^p}\,\|f\|^p_{L^p(\nu)}.
\end{equation*}

We will first show that
\begin{equation}
\label{fintp}
\int_{\R^d\setminus\bigcup_i 2B_i} |T_{\ve} \beta|^p\,d\mu \leq c \,\|f\|^p_{L^p(\nu)}.
\end{equation}
By duality
\begin{equation*}
\left(\int_{\R^d\setminus\bigcup_i 2B_i} |T_{\ve} \beta|^p\,d\mu \right)^{1/p}=\sup_{\substack{\supp(g)\subset \R^d\setminus\bigcup_i B_i,\\ \|g\|_{L^{p'}(\mu)}\leq 1 }}\left|\int_{\R^d\setminus\bigcup_i 2B_i} T_{\ve}(\beta) \ g d\mu\right|.
\end{equation*}
Then, for $g$ as above, we write
\begin{equation*}
\begin{split}
\left|\int_{\R^d\setminus\bigcup_i 2B_i} T_{\ve}(\beta) \, g \,d\mu\right| &\leq \sum_i \int_{\R^d \setminus 2 B_i}  |T_\ve\beta_i|\,|g|\, d \mu\\
&=\sum_i \int_{\R^d \setminus 2 R_i}  |T_\ve\beta_i|\,|g| \,d \mu+\sum_i \int_{2R_i \setminus 2 B_i}  |T_\ve(\beta_i)|\,|g|\, d \mu\\
&\leq \sum_i \int_{\R^d \setminus 2 R_i}  |T_\ve\beta_i|\,|g|\, d \mu+\sum_i \int_{2R_i \setminus 2 B_i}  |T_\ve(w_i f \nu \lfloor_{B_i})|\,|g|\, d \mu\\
&\quad+\sum_i \int_{2R_i \setminus 2 B_i}  |T_\ve(\varphi_i \hdg)|\,|g| \,d \mu \\&=: I + II+ III.
\end{split}
\end{equation*}
Therefore (\ref{fintp}) will follow if we prove that for any function $g$ such that $\supp(g)\subset \R^d\setminus\bigcup_i 2B_i$ with $\|g\|_{L^{p'}(\mu)}\leq 1$ we have $I+II+III\leq c  \|f\|^p_{L^p(\nu)}$.

To estimate $I$, using that, by (\ref{cc4}), $\beta_i(R_i)=\int_{B_i} w_i f d \nu - \int_\Gamma \vphi_i d  \hd=0$ and $\supp(\beta_i)\subset R_i,$ if $x_i$ stands for the center of $B_i$ and $y\notin 2 R_i$, we get
\begin{equation*}
\begin{split}
|T \beta_i(y)|&=\left|\int_{R_i}\bigl( k(y,x)-k(y,x_i)\bigr)\,d \beta_i(x)\right|\\
&\lesssim \int_{R_i}\frac{|x-x_i|^\eta}{|y-x_i|^{d-1+\eta}}d \beta_i(x)\leq \frac{r(R_i)^\eta\|\beta_i\|}{|y-x_i|^{d-1+\eta}}.
\end{split}
\end{equation*}  
Hence, using also that $\|\beta_i\|\leq2 \int_{B_i}|f|d \nu$, we obtain
\begin{equation}
\label{bife}
\begin{split}
\int_{\R^d \setminus 2 R_i}  |T_\ve\beta_i|\,|g|\, d \mu &\lesssim  \int_{\R^d \setminus 2 R_i}\frac{r(R_i)^\eta\|\beta_i\|}{|y-x_i|^{d-1+\eta}} |g(y)|d \mu (y) \\
&\lesssim  \int_{B_i}|f(x)|\left(\int_{\R^d \setminus 2 R_i} \frac{r(R_i)^\eta\,|g(y)|}{|y-x_i|^{d-1+\eta}} \,d \mu (y) \right)d \nu(x).
\end{split}
\end{equation} 

From the estimate \rf{eqfacil} applied to the measure $\mu'=|g|\mu$, 
taking into account that $\mu'(B(x,r)) \leq r^{d-1} M_R \mu'(x)$, we deduce that
$$\int_{\R^d \setminus 2 R_i} \frac{r(R_i)^\eta\,|g(y)|}{|y-x_i|^{d-1+\eta}} \,d \mu (y)\leq c\,M^\mu_Rg(x).$$
Hence by (\ref{bife}), H\"older's inequality, and (\ref{boundmax}),
\begin{equation*}
\begin{split}
I=\sum_i \int_{\R^d \setminus 2 R_i}  |T_\ve\beta_i|\,|g|\, d \mu &\leq \sum_i \int_{B_i}|f(x)|\, M^\mu_Rg(x)\,d \nu (x) \\
&\lesssim  \int |f(x)|\, M^\mu_Rg(x)\,d \nu (x) \\
&\leq \|f\|_{L^p(\nu)} \,\|M_R^\mu g\|_{L^{p'}(\nu)} \lesssim \|f\|_{L^p(\nu)} \,\|g\|_{L^{p'}(\mu)}
\leq \|f\|_{L^p(\nu)}.
\end{split}
\end{equation*}

To estimate $II$, notice that for $x\in 2R_i\setminus 2B_i$, 
$$|T_\ve(w_i f \nu \lfloor_{B_i})(x)|\leq  \|f\|_{L^{1}(\nu \lfloor_{B_i})}r(B_i)^{1-d}.$$
Therefore,
\begin{equation*}
\begin{split}
 \int_{2R_i \setminus 2 B_i}  |T_\ve(w_i f \nu \lfloor_{B_i})||g| d \mu &\leq \int_{2R_i \setminus 2 B_i} \|f\|_{L^{1}(\nu \lfloor_{B_i})}r(B_i)^{1-d}|g(x)|\,d \mu (x)\\
&\lesssim \int_{B_i}|f(y)|\left( r(R_i)^{1-d} \int_{2R_i}|g(x)|\, d \mu(x)\right)d \nu(y)\\
& \lesssim \int_{B_i}|f(y)| M^\mu_Rg(y)\,d \nu(y).
\end{split}
\end{equation*}
So using (\ref{boundmax}) we get:
\begin{equation*}
\begin{split}
II=\sum_i \int_{2R_i \setminus 2 B_i}  |T_\ve(w_i f \nu \lfloor_{B_i})||g| d \mu 
& \lesssim \sum_i \int_{B_i}|f(y)|\, M^\mu_Rg(y)\,d \nu(y)\\
&\lesssim \int |f(y)|\, M^\mu_Rg(y)\,d \nu(y) \\
& \leq \|f\|_{L^p(\nu)} \|M^\mu_Rg\|_{L^{p'}(\nu)}\\
& \lesssim \|f\|_{L^p(\nu)} \|g\|_{L^{p'}(\mu)}=\|f\|_{L^p(\nu)}.
\end{split}
\end{equation*}

We now turn our attention to the term $III$. Using H\"older's inequality for some $q <p'$ 
and the boundedness of $T_{\hdg}$ in $L^{q'}(\hdg)$,
 we get
\begin{equation*}
\begin{split}
 \int_{2R_i \setminus 2 B_i}  |T_\ve(\varphi_i \hdg)||g| d \mu &\leq \left( \int |T_{\hdg,\ve} \vphi_i|^{q'}d \mu\right)^{1/q'}\|g\ \chi_{2 R_i}\|_{L^q(\mu)}\\
&\lesssim \left( \int_{\Gamma} |\varphi_i |^{q'} d\hd\right)^{1/q'}\|g\ \chi_{2 R_i}\|_{L^q(\mu)}.
\end{split}
\end{equation*}
Using (\ref{cc6}), we obtain
\begin{equation*}
\begin{split}
\biggl( \int_{\Gamma} |\varphi_i |^{q'} d\hd \biggr)^{1/q'}\|g\, \chi_{2 R_i}\|_{L^q(\mu)} &\leq  
\|\vphi_i\|_{L^\infty(\hdg)}\,
 r(R_i)^{\frac{d-1}{q'}}\|g\ \chi_{2 R_i}\|_{L^q(\mu)} \\
&\lesssim  \frac{1}{r(R_i)^{d-1}}   
\int_{B_i}|f|d \nu\;
r(R_i)^{\frac{d-1}{q'}}\|g\ \chi_{2 R_i}\|_{L^q(\mu)}\\
&= \int_{B_i}|f(x)|\left( r(R_i)^{1-d} \int_{2 R_i}|g(y)|^q d \mu(y)\right)^{1/q}d \nu (x).
\end{split}
\end{equation*}
Notice that, for $x \in B_i$,
$$\left( r(R_i)^{1-d} \int_{2 R_i}|g(y)|^q d \mu(y)\right)^{1/q}\lesssim M^\mu_{R,q}(g)(x).$$
Combining all the above estimates and using H\"older's inequality we obtain
\begin{equation*}
\begin{split}
III=\sum_i \int_{2R_i \setminus 2 B_i}  &|T_\ve(\varphi_i \hdg)|\,|g|\,d \mu \leq \sum_i \int_{B_i}|f(x)|M^\mu_{R,q}g(x)d \nu (x)\\
&\lesssim \int |f(x)|M^\mu_{R,q}g(x)d \nu (x) \leq \|f\|_{L^p(\nu)}\|M^\mu_{R,q}g\|_{L^{p'}(\nu)}.
\end{split}
\end{equation*}
But for $p'>q$, $M^\mu_{R,q}$ is bounded in $L^{p'}(\nu)$, and thus
$$III\lesssim \|f\|_{L^p(\nu)}.$$
So \rf{fintp} is proved. 
By Chebyshev's inequality , this implies that
\begin{equation}
\label{eqqb} 
\mu\Bigl\{x\in \R^d\setminus \bigcup_i 2B_i:\,|T_{\ve}\beta(x)|>\lambda/2\Bigr\} \leq \frac{c}{\lambda^p} \int |T_{\ve}(\beta)|^p d \mu\leq \frac{c}{\lambda^p}\,\|f\|^p_{L^p(\nu)}.
\end{equation}

Now we have to estimate $\mu\{x \in \R^d: |T_\ve\kappa(x)|>\lambda/2\}$.
By Chebyshev's inequality,
\begin{equation}
\begin{split}
\label{sppoq1}
\mu\{x \in \R^d:& |T_\ve\kappa(x)|>\lambda/2\}\leq \frac{c}{\lambda^p}\int |T_{\ve}\kappa|^p d \mu\\
& \!\!\leq \frac{c}{\lambda^p}\left(\int |T_{\ve}(f\nu \lfloor_{\R^d \setminus \cup B_i})|^p d \mu
+\!
 \int \Bigl|T_{\hdg,\ve}\bigl( \sum_i \varphi_i\Bigr)\Bigr|^p d \mu\right).
 \end{split}
\end{equation}

Since $\nu \lfloor_\Gamma$ has $(d-1)$-growth and it is supported on $\Gamma$, it is absolutely continuous with respect to $\hdg$. Hence there exists some function $h',\ 0 \leq h'\lesssim 1$ $\hd$-a.e.\ in $\Gamma$ such that $\nu \lfloor_\Gamma=h' \hdg$. Furthermore, by Theorem \ref{czd2} $f\nu \lfloor_{\R^d \setminus \cup B_i}= h \hdg$ with $|h| \leq c \lambda$, $\hd$-a.e. in $\Gamma$. Thus $f\,h' \hdg=h \hdg$ and so\begin{equation}
\begin{split}
\label{shd1}
\int |T_{\ve}(f\nu \lfloor_{\R^d \setminus \cup B_i})|^p d \mu&=\int |T_{\hdg,\ve}(f\,h')|^p d \mu\\
&\lesssim \int_{\Gamma}|f|^p {h'}^p d \hd \lesssim \int_{\Gamma}|f|^p h'd \hd
=\int |f|^p d \nu.
\end{split}
\end{equation}

To estimate the last integral in \rf{sppoq1} we argue again by duality. Given any 
function $g\in L^{p'}(\hdg)$ with $\|g\|_{L^{p'}(\hdg)}\leq 1$, we have by (\ref{cc6})
\begin{equation*}
\begin{split}
\int_\Gamma |\varphi_i||g|d \hd &\leq   \int_{\Gamma\cap R_i} \|\vphi_i\|_{L^\infty(\hdg)} |g| d   \hd \\
&\leq \frac{1}{r(R_i)^{d-1}}\,\int_{B_i}|f|d \nu\; \int_{R_i}|g|d \hd 
\leq \int_{B_i} |f| M^{\hdg}_R(g)d \nu.
\end{split}
\end{equation*}
Therefore,
\begin{equation*}
\begin{split}
\int_\Gamma \sum_i |\varphi_i||g|d \hd &\lesssim \int |f| M^{\hdg}_R(g)d \nu 
\leq \|f\|_{L^p(\nu)} \| M^{\hdg}_R(g)\|_{L^{p'}(\nu)}\\
& \lesssim \|f\|_{L^p(\nu)}  \|g\|_{L^{p'}(\hdg)} \leq \|f\|_{L^p(\nu)}.
\end{split}
\end{equation*}
Hence
$$\int_\Gamma \bigl|\sum_i \varphi_i\bigr|^p d \hd \lesssim  \|f\|_{L^p(\nu)}.$$
Together with \rf{sppoq1} and \rf{shd1}, this yields
 
\begin{equation}
\label{final}
\mu\{x \in \R^d: |T_\ve\kappa(x)|>\lambda/2\}\leq \frac{c}{\lambda^p}\,\|f\|_{L^p(\nu)}
\end{equation}
Thus, when $\nu, \mu$ have compact support and $\supp\nu \subset N(\Gamma, \diam(\Gamma))$, (\ref{eqdebilp}) follows by (\ref{eqqb}), (\ref{sppoq1}), (\ref{shd1}) and (\ref{final}).

\vspace{2mm}
In the case that $\nu, \mu$ have compact support but $\supp\nu \not\subset N(\Gamma, \diam(\Gamma))$, we split $$f\nu=f\nu_1+f\nu_2:= f\nu \lfloor_{N(\Gamma, \diam(\Gamma))}+f\nu \lfloor_{\R^d \setminus N(\Gamma, \diam(\Gamma))}.$$
For $f\nu_1$ we already know that the estimate (\ref{eqdebilp}) holds. For $f\nu_2$, we
take into account that 
 $\dist(\supp \nu_2,\supp\mu)\geq \diam(\Gamma)$ and so, for $x\in\supp\mu$,
\begin{align*}
|T_{\ve}(f\nu_2)(x)|&\lesssim\int\frac{|f(y)|}{|x-y|^{d-1}}  d\nu_2(y)\\
&\leq \|f\|_{L^p(\nu_2)} \left(\int_{|x-y|>\diam(\Gamma)}\frac{1}{|x-y|^{p'(d-1)}}d\nu_2(y) \right)^{1/p'}
\\  
&\lesssim \|f\|_{L^p(\nu_2)}\,
 \frac{1}{\diam(\Gamma)^{(d-1)\frac{(p'-1)p}{p'}}}
 =\|f\|_{L^p(\nu_2)}\,
 \frac{1}{\diam(\Gamma)^{d-1}}.
 \end{align*}
Thus, using that $\|\mu\| \leq c_\mu (\diam \Gamma)^{d-1}$, 
\begin{equation*}
\begin{split}
\int |T_{\ve}(f\nu_2)|^p d \mu
&\lesssim \|f\|^p_{L^p(\nu)} \frac{\|\mu\|}{\diam(\Gamma)^{d-1}}\\
&\lesssim \|f\|^p_{L^p(\nu)}.
\end{split}
\end{equation*}
Therefore,
\begin{equation*}
\begin{split}  
\label{gg1}
\mu\Bigl\{x:\,|T_{\ve}(f\nu_2)(x)|&>\lambda \Bigr\}\leq \frac{1}{\lambda^p}\int |T_{\ve}(f\nu_2)|^p d \mu 
\lesssim \frac{1}{\lambda^p}\|f\|^p_{L^p(\nu)}.
\end{split}
\end{equation*}

\vspace{2mm}
Suppose now that $\mu$ is compactly supported but not $\nu$. 
Let $N_0$ be such that $\supp\mu\subset B(0,N_0)$, and for some $N>N_0$, let 
$\nu_N = \chi_{B(0,N)}\,\nu$. Then, for $x\in\supp\mu$,
$$|T_\ve(f(\nu-\nu_N))(x)|\leq \int_{\R^d \setminus  B(0,N)} |k(x,y)||f(y)|d \nu (y) \leq c\,\frac{\|f\|_{L^p(\nu)}\nu(\R^d\setminus B(0,N))^{1/p'}}{(N-N_0)^{d-1}}.$$
Thus $T_\ve\nu_N(x)\to T_\ve\nu(x)$ for all $x\in \supp\mu$, and since the estimate 
\rf{eqdebil49} holds
for $f\nu_N$, letting $N\to\infty$, we deduce that it also holds for $f\nu$.

On the other hand, if $\mu$ is not compactly supported, then for $\mu_N=\mu\lfloor B(0,N)$,
$$\mu_N\{x\in \R^d:\,|T_\ve(f\nu)(x)|>\lambda\}\leq c\frac{\|f\|^p_{L^p(\nu)}}{\lambda^p}$$
uniformly on $N$, and then \rf{eqdebil49} follows in full generality.
\end{proof}

\section{Remarks about the proof of Theorem \ref{davgenw}}

Following the same scheme as in the previous two sections we obtain the following theorem, which implies Theorem \ref{davgen} by interpolation.

\begin{theorem}
\label{wppm} 
Let $\mu, \nu$ two  Radon measures in $\R^d$ such that $\mu$ is $n$-AD regular and $\nu$ has $n$-growth. 
Let $k$ be an $n$-dimensional $CZ$-kernel such that $T^k_\mu$ is a bounded operator in $L^2(\mu)$.
 Then $T_\nu^k$ is bounded from $L^p(\nu)$ to
$L^{p,\infty}(\mu)$ and $T_\mu^k$ is bounded from $L^p(\mu)$ to
$L^{p,\infty}(\nu)$ for all $1<p<\infty$.
Moreover,
\begin{equation}
\label{w11meas}
T^k:M(\supp \mu)\rightarrow L^{1,\infty}(\nu)
\end{equation}
is also bounded.
\end{theorem}

For the proof of Theorem \ref{wppm} we need the following Calder\'on-Zygmund decomposition,
which is analogous to the one
from Theorem \ref{teodescz1}.

\begin{theorem}
\label{czd3} Let $\mu, \nu$ two measures with  compact support such that $\mu$ is $n$-AD regular, $\nu$ has $n$-growth and $\supp \nu \subset N(\supp \mu, \diam(\supp \mu))$. Then for all $f \in L^p(\nu), \ 1\leq p <\infty,$ and for all 
 $\lambda>2^{d+1}\,\|f\|^p_{L^p(\nu)}/\|\mu\|$:
\begin{itemize}
\item[(a)] There exists a family of
almost disjoint balls $\{B_i\}_i$ (that is, $\sum_i \chi_{B_i} \leq
c$) centered at $\supp \nu$, with radius not exceeding $3\,\diam(\supp\mu)$,
 and a function $h\in L^1(\mu)$ such that
\begin{equation}  \label{cc11}
 \int_{B_i}|f|^p d\nu >\frac{\lambda^p}{2^{d+1}}\,\mu(2B_i),
\end{equation}
\begin{equation}  \label{cc22}
\int_{\eta B_i}|f|^p d\nu \leq
\frac{\lambda^p}{2^{d+1}}\,\mu(2\eta B_i) \quad
\mbox{for $\eta >2$,}
\end{equation}
\begin{equation}  \label{cc33}
f \nu \lfloor_{\R^d\setminus\bigcup_i
B_i} = h\,\mu   \quad\mbox{with $|h|\leq c \lambda$ \; $\mu$ a.e..}
\end{equation}

\vv
\item[(b)] For each $i$, let $R_i$ be a ball concentric
with $B_i$, with $10 r(B_i)\leq r(R_i) \leq 30 \diam(\supp \mu)$ and denote
$w_i= \frac{\chi_{B_i}}{\sum_k \chi_{B_k}}$. Then,
there exists a family of functions
$\vphi_i$ with $\supp(\vphi_i)\subset R_i$ and with constant sign satisfying
\begin{equation}  \label{cc44}
\int \vphi_i \,d\mu = \int_{B_i} w_i f\,d\nu,
\end{equation}
\begin{equation}  \label{cc55}
\sum_i |\vphi_i| \leq c_1\,\lambda
\end{equation}
(where $c_1$ is some fixed constant), and
\begin{equation}  \label{cc66}
\|\vphi_i\|_{L^\infty(\mu)} \, \mu(R_i)\leq
c\, \int_{B_i}|f|d \nu.
\end{equation}

\end{itemize}
\end{theorem}

The proof of Theorem \ref{czd3} follows the same steps of the proof of Theorem \ref{czd2} and  it also makes 
essential use of the $AD$-regularity of the measure $\mu$. 
Then we can prove Theorem \ref{wppm} by following a strategy analogous to the one for
 Theorem \ref{wpp}. In every occasion where we needed to  
use the properties of the boundary $\Gamma$ in the proof of Theorem \ref{wpp}, for the proof of Theorem \ref{wppm} we  use now the $AD$-regularity of $\mu$.
We omit the details.

\section{Failure of Theorem \ref{davgen} for $\mu$ non AD-regular}

In the plane, we consider the so called corner quarters Cantor set. See Figure 1. This set is constructed in the following way: consider a square $Q^0$ with side length $1$. Now replace $Q^0$ by $4$ squares 
$Q^1_i$, $i=1,\ldots,4$, with side length $1/4$
contained in $Q^0$, so that each $Q^1_i$ contains a different vertex of $Q^0$.
Analogously, in the next stage each $Q^1_i$ is replaced by $4$ squares with side length $1/16$
contained in $Q^1_i$ so that each one contains a different vertex of $Q^1_i$. So we will have $16$ squares 
$Q^2_k$ of side length $1/16$. We proceed inductively (see Figure \ref{figcantor}), and we set $E_n= \bigcup_{i=1}^{4^n} 
Q^n_i$ and $E = \bigcap_{n=1}^\infty E_n$. This is the {corner quarters Cantor set}. It is not difficult to check that $0<\HH^1(E)<\infty$. 
 
\begin{figure} 
\begin{center}
\includegraphics[totalheight=20mm]{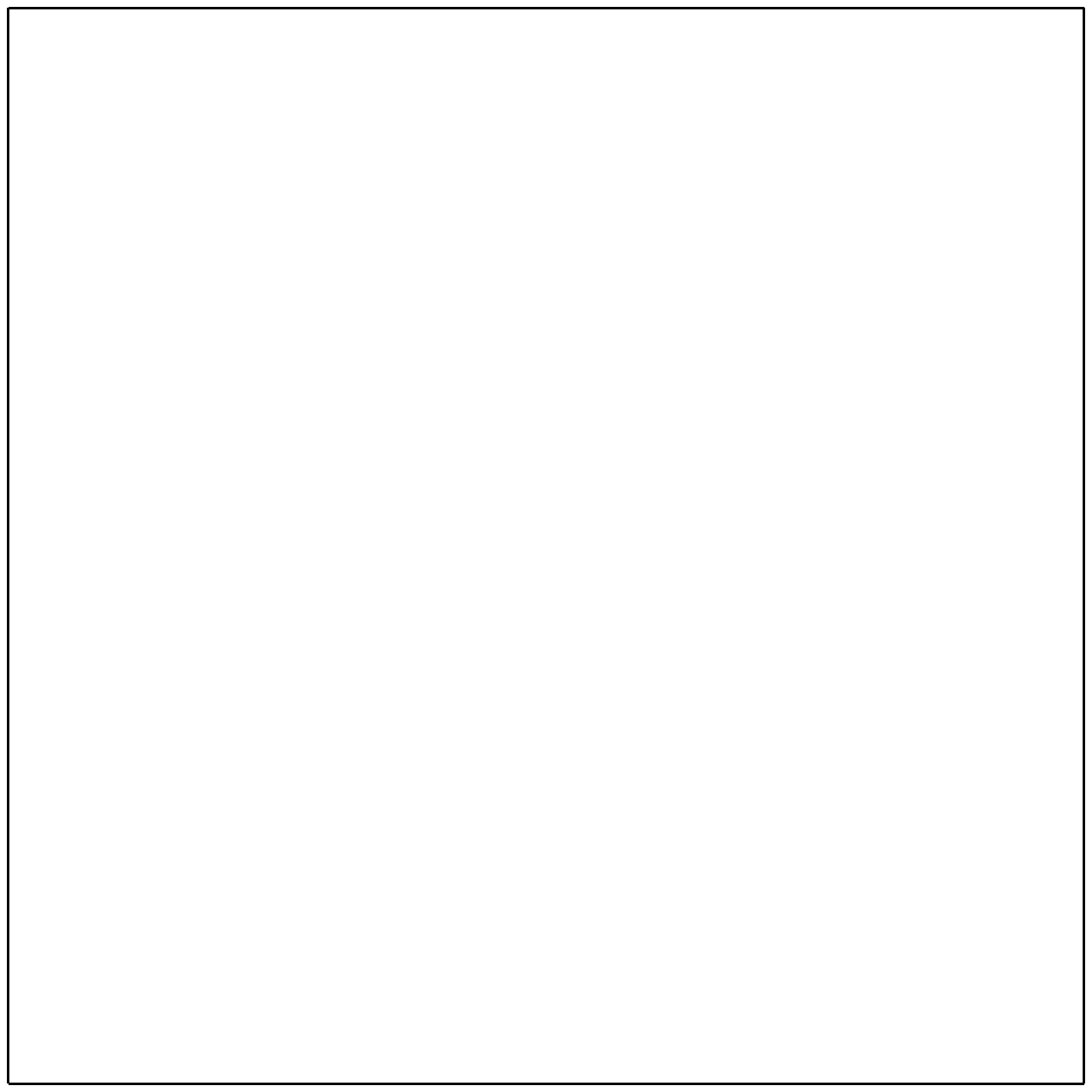}\hspace{18mm}
\includegraphics[totalheight=20mm]{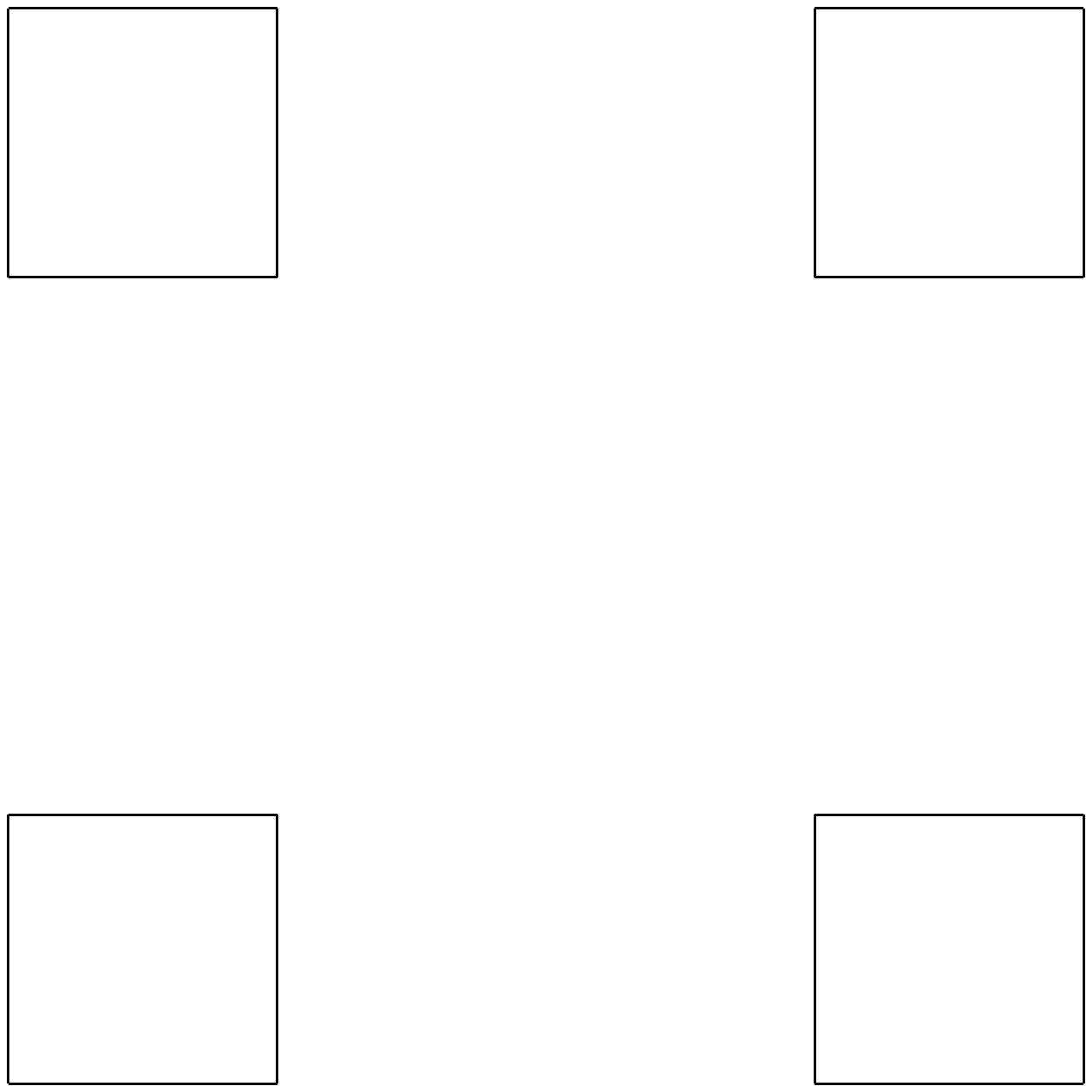}\hspace{18mm}
\includegraphics[totalheight=20mm]{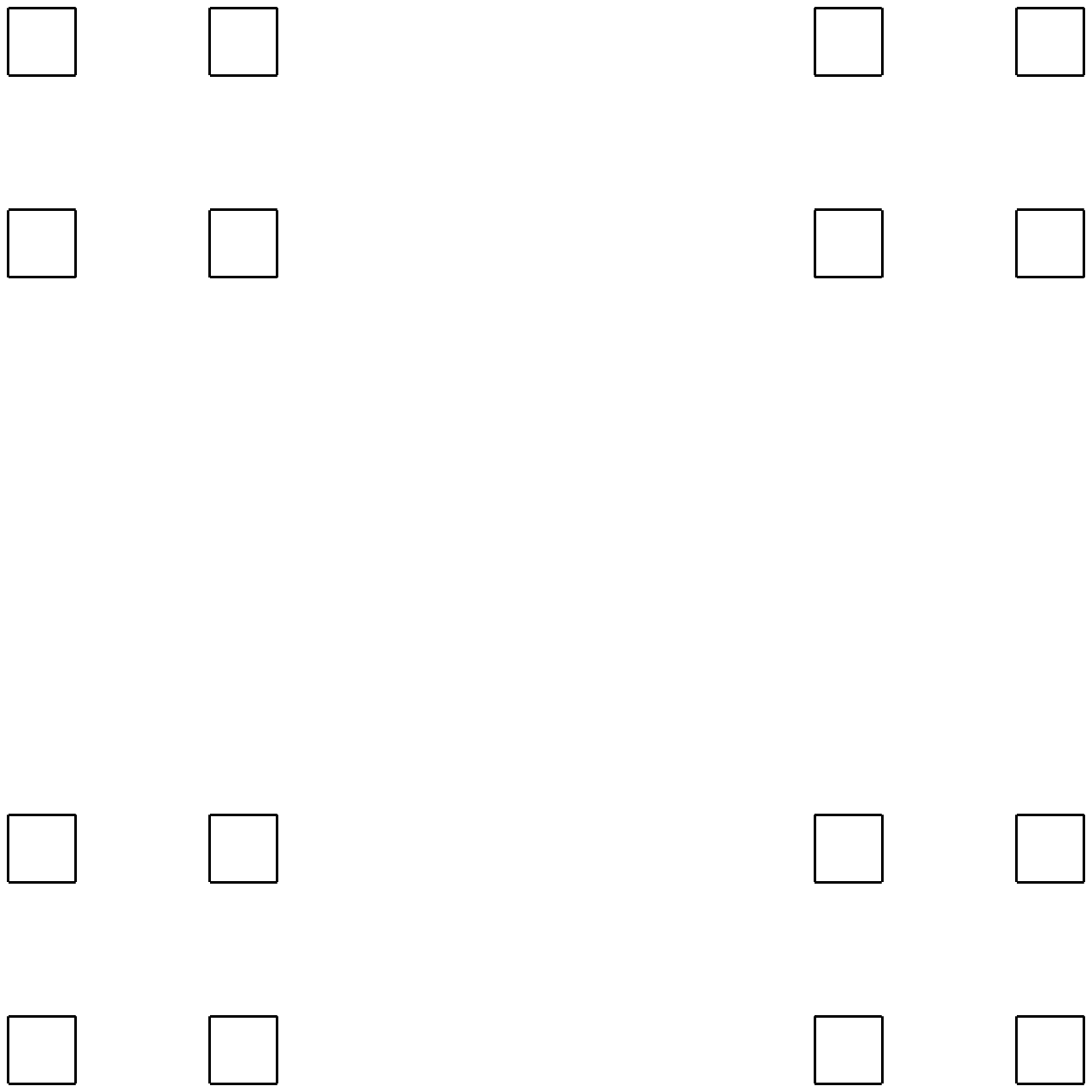}

\end{center}

\vspace{-1mm} {\scriptsize \hspace{3mm} $Q^0$ \hspace{31mm} $E_1$
\hspace{34mm} $E_2$}
\vspace{4mm}

\caption{The square $Q^0$ and the sets $E_1$ and $E_2$, which appear
in the first stages of the construction of the corner quarters Cantor set.
 \label{figcantor}}
\end{figure}

\vspace{3mm}

In the complex plane, consider now 
the Cauchy singular integral operator $\CC$, which is associated with the Cauchy kernel $1/z$.
It is well known that
$$\|\CC_{\HH^1\rest E_n}\|_{L^2(\HH^1\rest E_n)\to L^2(\HH^1\rest E_n)}\approx n^{1/2}.$$
See \cite{Mattila2}, for example.

Given a Borel measure $\sigma$ on $\C$ and $\lambda>0$, it is immediate to check that
$$\|\CC_{\lambda \sigma}\|_{L^2(\lambda\sigma)\to L^2(\lambda\sigma)} = \lambda\,\|\CC_{\sigma}\|_{L^2(\sigma)\to L^2(\sigma)}$$
and
$$\|\CC_{\lambda \sigma}\|_{L^2(\lambda\sigma)\to L^2(\sigma)} = \lambda^{1/2}\,\|\CC_{\sigma}\|_{L^2(\sigma)\to L^2(\sigma)}.$$
Also, these $L^2$ norms are invariant by translations. So, for any $z_n\in\C$ and $\lambda_n=n^{-1/2}$,
we consider the measures
$$\sigma_n=\HH^1\rest {(z_n+E_n)},\qquad \mu = \sum_{n\geq1} \lambda_n \,\sigma_n,
\qquad \nu = \sum_{n\geq1} \sigma_n.$$

Observe that if for each $n\geq1$,
$$\|\CC_{\lambda_n \sigma_n}\|_{L^2(\lambda_n\sigma_n)\to L^2(\lambda_n\sigma_n)}
\approx \lambda_n\,n^{1/2} =1.$$
Then, if the points $z_n$ are chosen far enough from each other, then the supports of the measures $\lambda_n\sigma_n$ will be very separated, and then
it easily follows that  
$\CC_\mu$ is bounded in $L^2(\mu)$. On the other hand,
$$\|\CC_{\mu}\|_{L^2(\mu)\to L^2(\nu)}
\geq \|\CC_{\lambda_n \sigma_n}\|_{L^2(\lambda_n\sigma_n)\to L^2(\sigma_n)}= \lambda_n^{1/2}n^{1/2}
=n^{1/4},$$
and thus $\CC_\mu$ is not bounded from $L^2(\mu)$ into $L^2(\nu)$.

\enlargethispage{2cm}

\end{document}